\documentclass[12pt]{amsart}
\usepackage{amsmath, amssymb, amsthm, bm, mathtools, enumitem, caption}
\usepackage{hyperref}

\usepackage[noabbrev,capitalize]{cleveref}
\usepackage{adjustbox}
\crefname{equation}{}{}
\usepackage{graphicx, tikz}
\usepackage[textheight=9.3in, textwidth=7in]{geometry}

\usepackage{microtype}

\newtheorem{theorem}{Theorem}[section]
\newtheorem{lemma}[theorem]{Lemma}

\newtheorem*{conjecture*}{Conjecture}

\theoremstyle{definition}

\theoremstyle{remark}
\newtheorem*{remark}{Remark}

\numberwithin{equation}{section}

\DeclarePairedDelimiter\abs{\lvert}{\rvert}

\newcommand{\R}{\mathbb R}
\newcommand{\N}{\mathbb N}

\newcommand{\re}{{\text {\rm Re}}}

\newcommand{\es}{\emptyset}

\newcommand{\ad}{a_{\Delta}}
 \def\H{\mathbb{H}}

\newcommand{\cA}{\mathcal A}

\newcommand{\cF}{\mathcal F}

\newcommand{\cP}{\mathcal P}

\newcommand{\C}{\mathbb C}
\newcommand{\SL}{\mathrm{SL}}

\newcommand{\Z}{\mathbb Z}

\renewcommand{\abs}[1]{\left\vert #1 \right \vert}

\begin{document}

\title[Zeros of Hecke polynomials arising from weak eigenforms]{Zeros of Hecke polynomials arising from weak eigenforms}

\thanks{2020 {\it{Mathematics Subject Classification.}} 11F30, 11F25}
\keywords{Fourier coefficients, Hecke operators, harmonic Maass forms, meromorphic modular forms}

\author{Kevin Gomez}
\address{Dept. of Mathematics, University of Virginia, Charlottesville, VA 22904}
\email{vhe4ht@virginia.edu}

\begin{abstract}
	We attach Hecke polynomials $P_n(F;x)$ to weak Hecke eigenforms $ F$ of weight $2-k$ and show that, for large $n$, every zero is simple and lies in $[0,1728]$. The construction pulls back a weakly holomorphic Hecke combination of $F$ along $j$; the analysis follows Hecke orbits on the unit-circle arc $\cA$, isolating a dominant “cosine” term and controlling the tail via Maass–Poincaré series and Whittaker/Bessel bounds. This extends the Rankin–Swinnerton-Dyer/Asai–Kaneko–Ninomiya picture from holomorphic forms to a broad class of harmonic Maass forms and yields a clean degree–monicity formula and simple criteria for zeros at $0$ and $1728$.
\end{abstract}
\maketitle
\noindent

\section{Introduction and Statement of Results}

For positive even integers $k$, the weight $k$ Eisenstein series $E_{k}(\tau)$ has Fourier expansion
$$
	E_{k}(\tau):= 1 -\frac{2k}{B_{k}}\sum_{n=1}^{\infty} \sigma_{k-1}(n)q^n,
$$
where $B_{k}$ is the $k^{{\text {\rm th}}}$ Bernoulli number, $\sigma_{\nu}(n):=\sum_{d\mid n}d^{\nu}$, and $q := e^{2\pi i \tau}$. If $k \geq 4$, then $E_k(\tau)$ is a weight $k$ holomorphic modular form on $\SL_2(\Z)$. Rankin and Swinnerton-Dyer \cite{RSD} showed that the zeros of these series lie on the lower boundary of the standard fundamental domain $\mathcal{F}$ for $\SL_2(\Z),$ the arc of the unit circle 
 \begin{equation} \label{A}
	\mathcal{A}:= \left \{ \tau \in \H \ : \ |\tau|=1\ \ {\text {\rm with}}\ \ -\frac{1}{2}\leq \re(\tau)\leq 0 \right\}.
\end{equation}
This result may be stated in terms of ``divisor polynomials'' of modular forms, with zeros in the interval $[0, 1728]$. Viewed in this lens, we observe similar phenomena in the zeros of polynomials arising from other forms on $\SL_2(\Z)$, including meromorphic Poincar\'e series \cite{Rankin} and images of the $j$-function under the Hecke action \cite{AKN}.

Here, we expand this picture to include a large class of harmonic Maass forms, generalizing work of Ono and the author \cite{GomezOno}. Namely, for weak Hecke eigenforms $F$ of weight $2-k$ with real Fourier coefficients, we define Hecke polynomials $P_n(F;x)$ and prove that, for large $n$, all of their zeros are simple and lie in $[0,1728]$. Our method expresses $F$ as a linear combination of Maass–Poincar\'e series and tracks Hecke action along the arc $\cA$.

To state these results, define
$$
\widetilde{E}_k(\tau):=\begin{cases}
 1 \ \ \ \ \ &{\text {\rm if}}\ k\equiv 0\pmod{12},\\
 E_{4}(z)^2E_6(z) \ \ \ \ \ &{\text {\rm if}}\ k\equiv 2\pmod{12},\\
 E_4(z)\ \ \ \ \ &{\text {\rm if}}\ k\equiv 4\pmod{12},\\
 E_6(z)\ \ \ \ \ &{\text {\rm if}}\ k\equiv 6\pmod{12},\\
 E_4(z)^2\ \ \ \ \ &{\text {\rm if}}\ k\equiv 8\pmod{12},\\
E_{4}(z)E_6(z) \ \ \ \ \ &{\text {\rm if}}\ k\equiv 10\pmod{12}.
\end{cases}
$$
We also recall the classical modular discriminant
$$
	\Delta(\tau)=\sum_{n=1}^{\infty} \tau(n)q^n:=\frac{E_4(\tau)^3-E_6(\tau)^2}{1728}=q-24q^2+252q^3-\cdots,
$$
and the modular $j$-invariant
$$
	j(\tau):=\frac{E_4(\tau)^3}{\Delta(\tau)}=q^{-1}+744+196884q+\cdots.
$$

We may then construct the Hecke polynomials $P_n(F;x)$.
\begin{theorem}\label{Theorem1}
	Suppose that $F(\tau) = \sum_{l=-m}^{\infty} c_F^+(l)q^l =  q^{-m} + O(q^{-m+1})$ is a weight $2 - k$ harmonic Maass weak eigenform with real Fourier coefficients with shadow $G(\tau) = \sum_{n=1}^{\infty} a(n)q^n$. Then the following are true.
	
	\noindent
	(1) For every $n \geq 2$, we have that
	$$
		H_n(F;\tau) := \Delta(\tau)^{b(k - 2)} \widetilde{E}_{k-2}(\tau) \cdot \left( n^{k-1} F(\tau) \ | \ T_{2-k}(n) - a(n)F(\tau) \right)
	$$
	is a weakly holomorphic function on $\SL_2(\Z)$, where
	$$
	b(k):=\begin{cases}
		 \lfloor k/12\rfloor \ \ \ \ \ &{\text {if}}\ k\not \equiv 2\pmod{12},\\
	           \lfloor k/12\rfloor -1 \ \ \ \ \ &{\text {if}}\ k\equiv 2\pmod{12}.
	\end{cases}
	$$
	
	\noindent
	(2) For every $n \geq 2$, there exists a monic polynomial $P_n(F;x) \in \R[x]$ of degree $mn - b(k - 2)$ for which
	$$
		P_n(F;j(\tau)) = H_n(F;\tau).
	$$
\end{theorem}

These polynomials, by their construction, have zeros naturally located in $[0, 1728]$ for sufficiently large $n$.
\begin{theorem} \label{Theorem2}
	Assuming the notations and hypotheses of Theorem~\ref{Theorem1}, if $n \geq 7$ satisfies
	$$
		C_F n^{k-1} e^{-\pi n\frac{\sqrt{3}}{2}} < 1,
	$$
	then $P_n(F;x)$ has $mn - b(k - 2)$ distinct zeros in $[0,1728]$, where
	$$
		C_F := \frac{1}{4}\max \left\{80m^{k-1}\sum_{l=1}^{m} \abs{c_F^+(-l)}, 1\right\}
	$$
	is a constant depending only on $F$.
\end{theorem}

\begin{remark}
	Moreover, we have that the zeros of $P_n(F;x)$ include $0$ (resp. $1728$) for $n$ satisfying the hypotheses of Theorem~\ref{Theorem2} if $k \equiv 2,4 \pmod{6}$ (resp. $k \equiv 2 \pmod{4}$).
\end{remark}

This paper is organized as follows. In Section 2, we discuss previous work concerning zeros of forms on $\SL_2(\Z)$ and recall necessary background on harmonic Maass forms. We prove Theorem~\ref{Theorem1} in Section 3 using the theory of Hecke operators for harmonic Maass forms. Finally, in Section 4, we recall and prove useful facts about Maass-Poincar\'e series and use these to prove Theorem~\ref{Theorem2}.

\section*{Acknowledgements}

The author thanks Ken Ono for very helpful guidance and comments in the writing of this manuscript.

\section{Background and Notation}

Theorems~\ref{Theorem1} and~\ref{Theorem2} are motivated by an array of similar results concerning the zeros of forms on $\SL_2(\Z)$. We begin with the theorem of Rankin and Swinnerton-Dyer, which in this context may be stated as follows. Given a modular form $f(\tau)$ of weight $k$, define
$$
	\widetilde{P}(f; j(\tau)) = \frac{f(\tau)}{\Delta(\tau)^{b(k)}\widetilde{E}_k(\tau)}.
$$
Since $\widetilde{E}_k(\tau)$ captures exactly the trivial zeros of $f(\tau)$, $\widetilde{P}(f; j(\tau))$ is holomorphic on the upper half-plane. By dividing by the appropriate power of $\Delta(\tau)$, we have a weakly holomorphic modular function (i.e. its poles are supported at the cusp at $\infty$). Because the $j$-function is a bijection from $\cF$ to $\C$, it follows that $\widetilde{P}(f; j(\tau))$ is a polynomial in $j(\tau)$. After noting that $j(i)=1728$ and $j(\omega)=0$, and then setting
$$
h_k(x):=\begin{cases}
  1 \ \ \ \ \ &{\text {\rm if}}\ k\equiv 0\pmod{12},\\
      x^2(x-1728) \ \ \ \ \ &{\text {\rm if}}\ k\equiv 2\pmod{12},\\
           x \ \ \ \ \ &{\text {\rm if}}\ k\equiv 4\pmod{12},\\
         x-1728 \ \ \ \ \ &{\text {\rm if}}\ k\equiv 6\pmod{12},\\
         x^2 \ \ \ \ \ &{\text {\rm if}}\ k\equiv 8\pmod{12},\\
               x(x-1728) \ \ \ \ \ &{\text {\rm if}}\ k\equiv 10\pmod{12},
 \end{cases}
$$
\noindent
we obtain the {\it divisor polynomial} of $f(\tau)$  by
 \begin{equation*}
 P(f;x) = h_k(x)\cdot \widetilde{P}(f;x).
 \end{equation*}
Since $j : \mathcal{A}\mapsto [0, 1728]$, the theorem of Rankin and Swinnerton-Dyer asserts that the zeros of $F(E_{2k};x)$ are in the interval $[0, 1728]$. Proceeding similarly, the zeros of certain meromorphic Poincar\'e series may also be pinpointed to the arc $\cA$ \cite{Rankin}.

The same phenomenon was discovered by Asai, Kaneko, and Ninomiya \cite{AKN} in connection with the Hecke operators $T_0(n)$ acting on the modular $j$-function. To describe their result, we recall an important sequence of modular functions. Let $j_0(\tau):=1$, and for every positive integer $n$, and let $j_n(\tau)$ be the unique modular function
which is holomorphic on $\H$ whose $q$-expansion satisfies
$$
	j_n(\tau)=q^{-n}+\sum_{l=1}^{\infty}c_n(l)q^l.
$$
Each $j_n(\tau)$ is a monic degree $n$ integer polynomial $J_n(x)$ for $x=j(\tau)$. We refer to these expressions as ``Hecke  polynomials'', as they satisfy
$$
	J_n(j(\tau))=j_n(\tau)= n(j_1(\tau) \ | \ T_0(n)).
$$
Asai, Kaneko, and Ninomiya proved that the zeros of $J_n(x) \in \Z[x]$ are simple and lie in $[0, 1728]$. This was accomplished by showing that the $J_n(x)$'s are well-approximated by a damped cosine wave for all $n \geq 2$.

Recently, in further pursuit of this phenomenon, Ono and the author \cite{GomezOno} proved an analogous result for the sequence of Hecke polynomials arising from the special weight $-10$ \textit{harmonic Maass form}
$$
	R(\tau) := M_\Delta(\tau) + N_\Delta(\tau).
$$
Here $N_\Delta(\tau)$ is the nonholomorphic period integral function
$$
	N_{\Delta}(\tau) = (2\pi)^{11}\cdot 11i\cdot \beta_{\Delta}\int_{-\overline{\tau}}^{i \infty}
\frac{\overline{\Delta(-\overline{z})}}{(-i(z+\tau))^{-10}} \ dz
$$
arising from the modular discriminant, with $\beta_\Delta$ an explicit real constant, and $M_\Delta(\tau)$ is the mock modular form
\begin{equation*}
\begin{split}
M_{\Delta}(\tau)&=\sum_{n=-1}^{\infty}\ad(n)q^n=11!\cdot q^{-1}+\frac{24\cdot 11!}{B_{12}}-73562460235.684\dots q
-929026615019.113\dots q^2 - \cdots.
\end{split}
\end{equation*}
The coefficients $a_\Delta(\tau)$ are given by absolutely convergent infinite series
\begin{equation*}
\ad(n)=-2\pi\Gamma(12)n^{-\frac{11}{2}}\cdot
                      \sum_{c=1}^{\infty}\frac{K(-1,n,c)}{c}\cdot I_{11}\left(\frac{4\pi \sqrt{n}}{c}
                      \right),
\end{equation*}
where $I_{11}(x)$ is the usual $I_{11}$-Bessel function and $K(m,n,c)$ is the Kloosterman sum
\begin{equation} \label{Kloosterman}
K(m,n,c):=
\sum_{v(c)^{\times}} e\left(\frac{m\overline
v+nv}{c}\right),
\end{equation}
where $v$ runs over the primitive residue classes modulo $c$, and $\overline v$ is the multiplicative inverse of $v$. For background on harmonic Maass forms and their applications, the reader can see \cite{BOAnnals, BOInventiones, BOPNAS, BOR, Br, BF, BruinierOno, OnoCDM, OnoMock}.

As seen from their construction (see \cite{OnoMock}), $M_\Delta(\tau)$ and $N_\Delta(\tau)$ are intrinsically related to the modular discriminant. Indeed, the action of the weight $-10$ Hecke operator on $N_\Delta(\tau)$ is essentially the weight $12$ Hecke action on $\Delta(\tau)$, which we explicate concretely as follows.  Let $\tau := u + iv$. Recall that if $F(\tau)$ is a harmonic Maass form of even weight $2 - k$, then it has a Fourier expansion of the form (for example, see Lemma 4.2 of \cite{HMF})
$$
	F(\tau) = \sum_{n \gg -\infty} c_F^+(n)q^n + \sum_{n < 0} c_F^-(n) \gamma(k - 1, -4\pi nv)q^n,
$$
where
$$
	\Gamma(s,z) := \int_z^\infty e^{-t} t^s \frac{dt}{t}
$$
is the incomplete gamma function. The $q$-series
$$
	F^+(\tau) := \sum_{n \gg -\infty} c_F^+(n)q^n
$$
is the \textit{holomorphic part} of $F(\tau)$, while its \textit{nonholomorphic part} is
$$
	F^-(\tau) := \sum_{n < 0} c_F^-(n) \gamma(k - 1, -4\pi nv)q^n.
$$
In the case of $R(\tau)$, we have that $F^+(\tau) = M_\Delta(\tau)$ and $F^-(\tau) = N_\Delta(\tau)$. We also recall the differential operator
$$
	\xi_k := 2iv^k \overline{\frac{\partial}{\partial \overline{\tau}}}.
$$

Following from the work Bruinier and Funke (see Proposition 3.2 of \cite{BF} and Theorem 5.9 of \cite{HMF}), the $\xi$-operator captures succinctly how we may relate a given cusp form to a corresponding infinite family of harmonic Maass forms. If $F(\tau)$ is a harmonic Maass form of weight $2 - k$, then $\xi_{2-k}(F(\tau))$ is a holomorphic cusp form of weight $k$, which we refer to as the \textit{shadow} of $F^+(\tau)$. Its Fourier expansion is given explicitly by (for example, see Theorem 5.9 of \cite{HMF})
\begin{equation} \label{HeckeXi}
	\xi_{2-k}(F(\tau)) = \xi_{2-k}(F^-(\tau)) = -(4\pi)^{k-1} \sum_{n=1}^{\infty} \overline{c_F^-(-n)}n^{k-1}q^n.
\end{equation}
For the case of $ R(\tau)$, we have that $\xi_{-10}(R(\tau)) = -11\beta_{\Delta} \cdot \Delta(\tau)$. Using the commutation relation (see Lemma 2.2 of \cite{GomezOno})
$$
	n^{k-1}\xi_{2-k}(F(\tau) \ | \ T_{2-k}(n)) = \xi_{2-k}(F(\tau)) \ | \ T_k(n),
$$
Ono and the author are able to construct a weight zero form with vanishing nonholomorphic part arising from $M_{\Delta}(\tau)$, and thus obtain a sequence of polynomials $\{P_n(R;x) \colon n \geq 2\}$ given by the relation

$$
	P_n(R;j(\tau)) = \frac{E_4(\tau)E_6(\tau)}{11!} \cdot \left(n^{11} M_\Delta(\tau) \ | \ T_{-10}(n) - \tau(n)M_\Delta(\tau)\right).
$$
For all $n \geq 2$, the zeros of these polynomials are simple and are in $[0,1728]$.

The case of $R(\tau)$ extends readily to all \textit{weak Hecke eigenforms} $F$; that is, weight $2-k$ harmonic Maass forms with real Fourier coefficients whose shadows are weight $k$ Hecke eigenforms. We mirror the methods in \cite{GomezOno} to show that $H_n(F;\tau)$ is a weight zero form with vanishing nonholomorphic part, and that the Hecke action along $\cA$ may be analyzed explicitly by realizing $H_n(F;\tau)$ as a linear combination of Maass-Poincar\'e series.

\section{Proof of Theorem~\ref{Theorem1}}

We now have the necessary background to prove Theorem~\ref{Theorem1}.

\begin{proof}[Proof of Theorem~\ref{Theorem1}]
Since $F$ is a weak Hecke eigenform, we have that, for each $n \geq 2$, $G = \xi_{2-k}(F)$ is an eigenform of $T_k(n)$ with eigenvalue $a(n)$. Via \eqref{HeckeXi}, we find that
$$
	n^{k-1}\xi_{2-k}(F(\tau) \ | \ T_{2-k}(n)) = \xi_{2-k}(F(\tau)) \ | \ T_k(n) = a(n)G(\tau) = a(n)\xi_{2-k}(F(\tau)).
$$
Thus, $F^-(\tau)$ is an eigenform of $T_{2-k}(n)$ with eigenvalue $n^{1-k}a(n)$. This implies that
$$
	n^{k-1}F^-(\tau) \ | \ T_{2-k}(n) - a(n)F^-(\tau) = 0.
$$
Therefore, we have that
$$
	n^{k-1}F(\tau) \ | \ T_{2-k}(n) - a(n)F(\tau) = n^{k-1}F^+(\tau) \ | \ T_{2-k}(n) - a(n)F^+(\tau)
$$
is a weakly holomorphic modular form on $\SL_2(\Z)$. We multiply by sufficient powers of $\Delta(\tau)$ and the form $\widetilde{E}_{k-2}(\tau)$ to obtain
$$
	H_n(F;\tau) =  \Delta(\tau)^{b(k-2)} \widetilde{E}_{k-2}(\tau) \cdot \left(n^{k-1}F^+(\tau) \ | \ T_{2-k}(n) - a(n)F^+(\tau)\right),
$$
which is a weakly holomorphic modular function since $F(\tau)$ is of weight $2 - k$. This proves claim (1).

To show (2), we expand $n^{k-1}F(\tau) \ | \ T_{2-k}(n)$ (see e.g. Proposition 2.1 of \cite{GomezOno}) as
$$
	n^{k-1}F(\tau) \ | \ T_{2-k}(n) = n^{k-1} \sum_{n' \leq 0} \sum_{\substack{d \mid (n,n') \\ d > 0}} d^{1-k} c_F\left( \frac{nn'}{d^2}\right)q^{n'} + O(q).
$$
Since $c_F(nn'/d^2) = 0$ whenever $nn'/d^2 < -m $, the principal part of $H_n(F;\tau)$ only contains terms for which $nn' \geq -m  d^2$. In particular, since $d \leq n$, we find that
$$
	n^{k-1}F(\tau) \ | \ T_{2-k}(n) = q^{-mn} + n^{k-1} \sum_{-mn < n' \leq 0} \sum_{\substack{d \mid (n,n') \\ d > 0}} d^{1-k} c_F\left( \frac{nn'}{d^2}\right)q^{n'} + O(q).
$$
As $F(\tau) = q^{-m} + O(q^{-m+1})$, we see that $n^{k-1} F(\tau) \ | \ T_{2-k}(n) = q^{-mn} + O(q^{-mn+1})$. Multiplying by $\widetilde{E}_{k-2}(\tau) = 1 + O(q)$ and $\Delta^{b(k-2)}(\tau) = q^{b(k-2)} + O(q^{b(k-2)+1})$, claim (1) grants that $H_n(F;\tau)$ is a monic polynomial in $j(\tau)$ of degree $mn - b(k - 2)$, proving (2).
\end{proof}

\section{Proof of Theorem~\ref{Theorem2}}
	
The proof of Theorem~\ref{Theorem2} relies on two infinite sequences of Poincar\'e series (see \cite{HMF}, \cite{OnoCBMS}). Given a positive integer $l$, we first require the classical Poincar\'e series of exponential type given by
\begin{equation}\label{PoincareDefinition}
	\cP_{k,l}(\tau) := \sum_{M \in \Gamma_{\infty} \backslash \SL_2(\Z)} (\varphi_l \ |_k \ M)(\tau),
\end{equation}
where $\Gamma_{\infty} := \{\pm \left(\begin{smallmatrix}1 & n \\ 0 & 1\end{smallmatrix}\right), n \in \Z\}$ is the group of translations and $\varphi_l(\tau) := q^l$.

We then require the Maass-Poincar\'e series
\begin{equation}\label{MaassPoincareDefinition}
	\cF_{2-k,-l}(\tau) := \sum_{M \in \Gamma_{\infty} \backslash \SL_2(\Z)} (\phi_{-l} \ |_{2-k} \ M)(\tau), 
\end{equation}
where
$$
	\phi_{-l}(\tau) := \frac{(4 \pi l v)^{\frac{k-2}{2}}}{\Gamma(k)}  M_{\frac{k-2}{2},\frac{k-1}{2}}(4\pi l v) e^{-2\pi i l u},
$$
and $M_{\kappa,\mu}$ is the usual $M$-Whittaker function. Here we recall the basic properties of these Poincar\'e series and how they relate to one another, which includes their Fourier expansions; these are given in terms of $I$-Bessel and $J$-Bessel functions and the Kloosterman sums defined in \eqref{Kloosterman}.

\begin{lemma}[Theorem 6.10 of \cite{HMF}] \label{MaassPoincare}
	Assuming the hypotheses above, the following are true.

	\noindent
	(1) We have
	$$
		\xi_{2-k}(\cF_{2-k,-l}) = (-1)^{k+1} \frac{(4\pi l)^{k-1}}{\Gamma(k-1)} \cP_{k,l}.
	$$
	
	\noindent
	(2) We have that $\cF_{2-k,-l}(\tau)$ is a weight $2-k$ harmonic Maass form on $\SL_2(\Z)$, with
	$$
	\cF_{2-k,-l}(\tau) = 	\cF_{2-k,-l}^+(\tau) + 	\cF_{2-k,-l}^-(\tau),
	$$
	where $\cF_{2-k,-l}^+(\tau)$ (resp. $\cF_{2-k,-l}^-(\tau)$) is its holomorphic part (resp. nonholomorphic part).
	
	\smallskip
	\noindent
	(3)	 The holomorphic part of $\cF_{2-k,-l}(\tau)$ has Fourier expansion
	$$
	\cF_{2-k,-l}^+(\tau)= q^{-l} - \frac{(2\pi i)^k l^{k-1}}{\Gamma(k)}\sum_{c > 0} \frac{K_{2-k}(l,0;c)}{c^k} + \sum_{n=1}^{\infty} c^+_{2-k,-l}(n)q^n,
	$$
	where for positive integers $n$ we have
	$$
		c_{2-k,-l}^+(n) = 2\pi i^{2-k} \left(\frac{l}{n}\right)^{\frac{k-1}{2}}  \sum_{c > 0} \frac{K_{2-k}(l,n;c)}{c} \cdot I_{k-1}\left(\frac{4\pi\sqrt{ln}}{c}\right).
	$$	
	
	\smallskip
	\noindent
	(4) The nonholomorphic part of $\cF_{2-k,-l}(\tau)$ has Fourier expansion
	$$
	\cF_{2-k,-l}^-(\tau) = \sum_{n=1}^{\infty} c_{2-k,-l}^-(-n)\Gamma(1 - k, 4\pi n v)q^{-n},	
	$$
	where for positive integers $n$ we have
	$$
		c_{2-k,-l}^-(-n) = 2\pi i^{2-k} \left(\frac{l}{n}\right)^{\frac{k-1}{2}}  \sum_{c > 0} \frac{K_{2-k}(l,-n;c)}{c} \cdot J_{k-1}\left(\frac{4\pi\sqrt{l n}}{c}\right).
	$$
\end{lemma}

Using these Poincar\'e series, we are able to construct any cusp form or harmonic Maass form.

\begin{lemma} \label{PoincareBasis}
	Assuming the notation and hypotheses above, the following are true.

	\noindent
	(1) Every cusp form of weight $k$ is a linear combination of Poincar\'e series $\cP_{k,l}(\tau)$ with $l \geq 1$.
	
	\noindent
	(2) Every harmonic Maass form of weight $2-k$ is a linear combination of Poincar\'e series $\cF_{2-k,-l}(\tau)$ with $l \geq 1$. In particular, if $F(\tau) = \sum_{l=-m}^{\infty} c_F^+(l)q^l$, then
	$$
		F(\tau) = \sum_{l=1}^{m} c_F^+(-l) \cF_{2-k,-l}(\tau).
	$$
\end{lemma}

\begin{proof}
	Claim (1) is a standard result in the theory of modular forms (see, for example, Theorem 6.7(iii) of \cite{HMF}). To prove (ii), we observe that
	$$
		F^+(\tau) - \sum_{l=1}^{m} c_F^+(-l)\cF_{2-k,-l}(\tau) = O(1)
	$$
	by Lemma~\ref{MaassPoincare} (3). By Lemma 5.12 of \cite{HMF}, this implies that $F^+(\tau) - \sum_{l=1}^{m} c_F(-m)\cF_{2-k,-l}(\tau)$ is either identically zero or a weight $2-k$ holomorphic form. Since no such forms exist, we obtain claim (2).
\end{proof}

We now describe the Hecke action on these Poincar\'e series.
\begin{lemma}
	Assuming the notation and hypotheses above, we have, for all $n \geq 1$,
	$$
		\cF_{2-k,-l}(\tau) \ | \ T_{2-k}(n) = \sum_{d \mid (l,n)} \left( \frac{n}{d} \right)^{1-k} \cF_{2-k,-\frac{ln}{d^2}}(\tau).
	$$
\end{lemma}

\begin{proof}
	This follows readily from the application of Lemma~\ref{MaassPoincare} (1) and Proposition 2.1 of \cite{GomezOno} to the identity
	$$
		\cP_{k,l}(\tau) \ | \ T_{k}(n) = \sum_{d \mid (l,n)} \left( \frac{n}{d} \right)^{k-1} \cP_{k,\frac{ln}{d^2}}(\tau).
	$$
\end{proof}

Thanks to these lemmas, the proof of Theorem~\ref{Theorem2} is reduced to the study of the zeros of these Poincar\'e series.  For the sequel, we let $\tau = e^{i\theta},$ where $\theta \in [\frac{\pi}{3},\frac{\pi}{2}],$ so that $\tau \in \cA$. Let $F$ be a weak Hecke eigenform of weight $2-k$ with real Fourier coefficients, which we may write by Lemma~\ref{PoincareBasis} as
$$
	F(\tau) = \sum_{l=1}^{m} c_F^+(-l) \cF_{2-k,-l}(\tau).
$$
We have
$$
	F(\tau) \ | \ T_{2-k}(n) = \sum_{l=1}^{m} c_F^+(-l) \sum_{d \mid (l,n)} \left( \frac{n}{d} \right)^{1-k} \cF_{2-k,-\frac{ln}{d^2}}(\tau),
$$
and thus
\begin{equation} \label{FullExpansion}
	H_n^\ast(F;\tau) = \sum_{l=1}^{m}  c_F^+(-l) \left(-a(n) \cF_{2-k,-l}(\tau) + \sum_{d \mid (l,n)} d^{k-1} \cF_{2-k,-\frac{ln}{d^2}}(\tau)\right),
\end{equation}
where $\Delta(\tau)^{b(k-2)}\widetilde{E}_{k-2}(\tau)H_n^\ast(F;\tau) = H_n(F;\tau)$.

We aim to show that $H_n^\ast(F;\tau)$ is well-approximated by a damped cosine wave which controls the location of its zeros, and hence the location of the zeros of $H_n(F;\tau)$. Indeed,  if $f$ is a weakly holomorphic form of weight $2-k$ with real Fourier coefficients, then $f(-1/e^{i\theta}) = \overline{f(e^{i\theta})}$, wherein
$$
	\overline{e^{\frac{2-k}{2}i\theta}f(e^{i\theta})} = 	e^{-\frac{2-k}{2}i\theta}\overline{f(e^{i\theta})} = e^{-\frac{2-k}{2}i\theta}f(-1/e^{i\theta}) = e^{\frac{2-k}{2}i\theta}f(e^{i\theta})
$$
by modularity, implying that $e^{\frac{2-k}{2}i\theta}f(e^{i\theta})$ is real (see Proposition 2.1 of \cite{Getz}). Since $F$ has real Fourier coefficients by assumption, and the eigenvalues of a Hecke eigenform on $\SL_2(\Z)$ are also real, $H_n^\ast(F;\tau)$ has only real coefficients, and thus we specifically aim to show
\begin{equation} \label{Goal}
	\abs{e^{\frac{2-k}{2}i\theta}e^{-2\pi mn \sin \theta} H_n^\ast(F;e^{i\theta}) - f_{mn}(\theta)} < 2
\end{equation}
for sufficiently large $n$, where, for all $l \geq 1$,
$$
	f_l(\theta) := 2\left( 1 - e^{-4\pi l\sin \theta} e_{k-2}(4\pi l \sin \theta)\right) \cos\left(\frac{k-2}{2}\theta + 2\pi l \cos \theta \right)
$$
and
$$
	e_j(x) := \sum_{n=0}^{j} \frac{x^n}{n!}
$$
is the $j$th order Taylor approximation of $e^x$.

To accomplish this, we first isolate the main term in \eqref{FullExpansion} and write
\begin{equation} \label{MainTerm}
	H_n^\ast(F;e^{i\theta}) = \cF_{2-k,-mn}(e^{i\theta}) + R_n(F;\theta),
\end{equation}
where
$$
	R_n(F;\theta) := \sum_{l=1}^{m} c_F^+(-l) \left(-a(n) \cF_{2-k,-l}(e^{i\theta}) + \sum_{\substack{d \mid (l,n) \\ (d,l) \neq (1,m)}} d^{k-1} \cF_{2-k,-\frac{ln}{d^2}}(e^{i\theta})\right).
$$

To analyze the contributions from each series, we employ \eqref{MaassPoincareDefinition} to write
$$
	\cF_{2-k,-l}(e^{i\theta}) = \sum_{\substack{c \geq 0, d \in \Z \\ (c,d)=1}} (ce^{i\theta} + d)^{k-2} \phi_{-l}\left( \frac{ae^{i\theta} + b}{ce^{i\theta} + d} \right),
$$
where $a,b \in \Z$ is an arbitrary solution to $ad - bc = 1$. We extract the main terms corresponding to $(c,d) = (0,1)$ and $(1,0)$, obtaining
\begin{equation} \label{MainTerms}
	\cF_{2-k,-l}(e^{i\theta}) = \phi_{-l}(e^{i\theta}) + e^{\frac{2-k}{2}i\theta}\phi_{-l}(-e^{-i\theta}) + Q_{2-k,-l}(\theta),
\end{equation}
where
$$
	Q_{2-k,-l}(\theta) := \sum_{\substack{c \geq 1, d \in \Z \setminus\{0\}\\ (c,d)=1}} (ce^{i\theta} + d)^{k-2} \phi_{-l}\left( \frac{ae^{i\theta} + b}{ce^{i\theta} + d} \right).
$$

We rewrite the main terms in \eqref{MainTerms} using the definition of $\phi_{-l}$, which yields
\begin{equation} \label{MainTermsSimp}
	\phi_{-l}(e^{i\theta}) + e^{\frac{2-k}{2}i\theta}\phi_{-l}(-e^{-i\theta}) = \frac{(4\pi l\sin \theta)^{\frac{k-2}{2}}}{\Gamma(k)}M_{\frac{k-2}{2},\frac{k-1}{2}}(4\pi l \sin \theta)\left(e^{2\pi i l \cos \theta} + e^{\frac{2-k}{2}i\theta}e^{-2\pi i l \cos \theta}\right).
\end{equation}
By (13.18.4) and (8.4.7) of \cite{DLMF}, we have, for all $\kappa \in \N$ and $x \in \R$, that
\begin{equation} \label{MSimp}
	M_{\kappa,\kappa+\frac{1}{2}}(x) = (2\kappa + 1)! \cdot \frac{e^{\frac{x}{2}} - e^{-\frac{x}{2}}e_{2\kappa}(x)}{x^{\kappa}}.
\end{equation}
We multiply \eqref{MainTermsSimp} by $e^{\frac{2-k}{2}i\theta}e^{-2\pi l \sin \theta}$ and apply \eqref{MSimp}; this simplifies to $f_l(\theta)$. We thus have that
\begin{equation} \label{QR}
	\abs{e^{\frac{2-k}{2}i\theta}e^{-2\pi mn \sin \theta}H_n^\ast(F;e^{i\theta}) - f_{mn}(\theta)} = e^{-2\pi m\sin \theta}\abs{Q_{2-k,-mn}(\theta) + R_n(F;\theta)}.
\end{equation}

From here, we require a number of lemmas to establish \eqref{Goal}. We first bound the $Q_{2-k,-l}(\theta)$ term.

\begin{lemma} \label{QBound}
	If $l \geq 1$ and $\theta \in [\frac{\pi}{3},\frac{\pi}{2}]$, then we have
	$$
		\abs{Q_{2-k,-l}(\theta)} \leq e^{2\pi l \sin \theta} + 8le^{\pi l \sin \theta} + 1.008 \times 10^{8}l^{k-1}.
	$$
\end{lemma}

\begin{proof}
	By the triangle inequality, we have
	$$
		\abs{Q_{2-k,-l}(\theta)} \leq \sum_{\substack{c \geq 1, d \in \Z \setminus\{0\}\\ (c,d)=1}} \abs{ce^{i\theta} + d}^{k-2} \abs{\phi_{-l}\left( \frac{ae^{i\theta} + b}{ce^{i\theta} + d} \right)}.
	$$
	Substituting the definition of $\phi_{-l}$ yields
	$$
		\sum_{\substack{c \geq 1, d \in \Z \setminus\{0\}\\ (c,d)=1}} \abs{ce^{i\theta} + d}^{k-2} \abs{\phi_{-l}\left( \frac{ae^{i\theta} + b}{ce^{i\theta} + d} \right)} = \frac{(4\pi l \sin \theta)^{\frac{k-2}{2}}}{\Gamma(k)} \sum_{\substack{c \geq 1, d \in \Z \setminus\{0\}\\ (c,d)=1}} M_{\frac{k-2}{2},\frac{k-1}{2}}\left(\frac{4\pi l \sin \theta}{\abs{ce^{i\theta} + d}^2}\right).
	$$
	We have that $M_{\frac{k-2}{2},\frac{k-1}{2}}$ is monotonically increasing on the real line. Furthermore, $\abs{ce^{i\theta}+d}^2 = c^2 + 2cd\cos \theta + d^2 \leq c^2 + cd + d^2$ for $\theta \in [\frac{\pi}{3},\frac{\pi}{2}]$, whereby
	$$
		\sum_{\substack{c \geq 1, d \in \Z \setminus \{0\} \\ (c,d)=1}} M_{\frac{k-2}{2},\frac{k-1}{2}}\left(\frac{4\pi l \sin \theta}{\abs{ce^{i\theta} + d}^2}\right) \leq \sum_{\substack{c \geq 1, d \in \Z \setminus \{0\} \\ (c,d)=1}} M_{\frac{k-2}{2},\frac{k-1}{2}} \left(\frac{4\pi l \sin \theta}{c^2 + cd + d^2}\right).
	$$
	
	We separate the right-hand sum into two parts. If $l \leq c^2 + cd + d^2$, then
	$$
		\frac{4\pi l \sin \theta}{c^2 + cd + d^2} \leq 4\pi \sin \theta \leq 4\pi,
	$$
	so Lemma 2.5 of \cite{GomezOno} gives
	$$
		M_{\frac{k-2}{2},\frac{k-1}{2}}\left(\frac{4\pi l \sin \theta}{c^2 + cd + d^2}\right) \leq e^{2\pi}(4\pi l \sin \theta)^{\frac{k}{2}} \frac{1}{(c^2 + cd + d^2)^{\frac{k}{2}}}.
	$$
	Otherwise, by \eqref{MSimp}, we have $M_{\frac{k-2}{2},\frac{k-1}{2}}(x) \leq (k-1)! \cdot x^{\frac{2-k}{2}}e^{\frac{x}{2}}$. We first have the case of $(c,d) = (1,-1)$, which contributes the term
	\begin{equation} \label{phil}
		\frac{(4\pi l \sin \theta)^{\frac{k-2}{2}}}{\Gamma(k)} M_{\frac{k-2}{2},\frac{k-1}{2}}(4\pi l \sin \theta) \leq e^{2\pi l \sin \theta}.
	\end{equation}
	For all other pairs $(c,d)$, we have that $c^2 + cd + d^2 \geq 2$, so
	$$
		\frac{(4\pi l \sin \theta)^{\frac{k-2}{2}}}{\Gamma(k)} M_{\frac{k-2}{2},\frac{k-1}{2}}\left(\frac{4\pi l \sin \theta}{c^2 + cd + d^2}\right) \leq e^{\pi l \sin \theta}.
	$$

	Therefore, we conclude that
	$$
		\abs{Q_{2-k,-l}(\theta)} \leq e^{2\pi l \sin \theta} + \sum_{\substack{c \geq 1, d \in \Z \\ 1 < c^2 + cd + d^2 < l}} e^{\pi l \sin \theta} + \frac{e^{2\pi}(4\pi l \sin \theta)^{k-1}}{\Gamma(k)}\sum_{\substack{c \geq 1, d \in \Z \\ c^2 + cd + d^2 \geq l}} \frac{1}{(c^2 + cd + d^2)^{\frac{k}{2}}}.
	$$
	Those pairs $(c,d)$ with $1 < c^2 + cd + d^2 < l$ satisfy $\max\{\abs{c},\abs{d}\} < \sqrt{3l}$, so there are at most $(\sqrt{3l} + 1)^2 \leq 8l$ terms in the left-hand sum. Meanwhile, $k \geq 12$ since $F$ is a weak Hecke eigenform, and thus the right-hand sum is bounded above by the Epstein zeta value
	$$
		\sum_{\substack{(c,d) \in \Z \\ (c,d) \neq (0,0)}} \frac{1}{(c^2 + cd + d^2)^{\frac{k}{2}}} \leq \sum_{\substack{(c,d) \in \Z \\ (c,d) \neq (0,0)}} \frac{1}{(c^2 + cd + d^2)^6} \leq 6.0099.
	$$
	Putting this all together, we find that
	$$
		\abs{Q_{2-k,-l}(\theta)} \leq e^{2\pi l \sin \theta} + 8le^{\pi l \sin \theta} + \frac{3219(4\pi l)^{k-1}}{\Gamma(k)}.
	$$
	Then, by Stirling's approximation, we find that
	$$
		\frac{(4\pi)^{k-1}}{\Gamma(k)} \leq 31294
	$$
	for all even $k \geq 12$. Thus,
	$$
		\abs{Q_{2-k,-l}(\theta)} \leq e^{2\pi l \sin \theta} + 8le^{\pi l \sin \theta} + 1.008 \times 10^{8}l^{k-1}
	$$
	and the lemma is proved.
\end{proof}

We then describe the contribution of the main term in \eqref{MainTerm}.
\begin{lemma} \label{PBound}
	For all $l \geq 1$ and $\theta \in [\frac{\pi}{3},\frac{\pi}{2}]$, we have
	$$
		\abs{\cF_{2-k,-l}(e^{i\theta})} \leq 4e^{2\pi l \sin \theta} + 1.008 \times 10^8 l^{k-1}.
	$$
\end{lemma}

\begin{proof}
	We first have, via \eqref{MainTerms} and the triangle inequality,
	$$
		\abs{\cF_{2-k,-l}(e^{i\theta})} \leq \abs{\phi_{-l}(e^{i\theta})} + \abs{\phi_{-l}(-e^{i\theta})} + \abs{Q_{2-k,-l}(\theta)}.
	$$
	The definition of $\phi_{-l}$ gives
	$$
		\abs{\cF_{2-k,-l}(e^{i\theta})} \leq \frac{2(4\pi l \sin \theta)^{\frac{k-2}{2}}}{\Gamma(k)} M_{\frac{k-2}{2},\frac{k-1}{2}}(4\pi l \sin \theta) + \abs{Q_{2-k,-l}(\theta)}.
	$$
	Following \eqref{phil} and applying Lemma \ref{QBound}, we find that
	$$
		\abs{\cF_{2-k,-l}(e^{i\theta})} \leq 3e^{2\pi l \sin \theta} + 8le^{\pi l \sin \theta} + 1.008 \times 10^8 l^{k-1}.
	$$
	We then observe that $8le^{\pi l \sin \theta} \leq e^{2\pi l \sin \theta}$ for all $l \geq 1$, wherein we obtain our desired result.
\end{proof}

Finally, we bound the remaining terms in \eqref{MainTerm}.
\begin{lemma} \label{RBound}
	For all $n \geq 2$ and $\theta \in [\frac{\pi}{3},\frac{\pi}{2}]$, we have
	$$
		\abs{R_n(F;\theta)} \leq\left(  8m^{k-2}n^{\frac{k}{2}}e^{\pi(2m - 1)n\sin \theta} + 2.016 \times 10^8 (mn)^{k-1}\right) \sum_{l=1}^{m } \abs{c_F^+(-l)}
	$$
\end{lemma}

\begin{proof}
	We have, by the triangle inequality,
	$$
		\abs{R_n(F;\theta)} \leq \sum_{l=1}^{m}  \abs{c_F^+(-l)} \left(\abs{a(n) \cF_{2-k,-l}(e^{i\theta})} + \sum_{\substack{d \mid (l,n) \\ (d,l) \neq (1,m)}} d^{k-1} \abs{\cF_{2-k,-\frac{l n}{d^2}}(e^{i\theta})}\right).
	$$
	First, Lemma \ref{PBound} grants
	\begin{equation} \label{TriR}
		\sum_{\substack{d \mid (l,n) \\ (d,l) \neq (1,m )}} d^{k-1} \abs{\cF_{2-k,-\frac{ln}{d^2}}(e^{i\theta})} \leq 4\sum_{\substack{d \mid (l,n) \\ (d,l) \neq (1,m)}} d^{k-3} e^{2\pi \frac{ln}{d^2}\sin \theta} + 1.008 \times 10^8 (ln)^{k-1} \sum_{\substack{d \mid (l,n) \\ (d,l) \neq (1,m)}} d^{1-k}.
	\end{equation}
	We analyze the righthand sums in two cases. For $l = m$, we have
	$$
		4\sum_{\substack{d \mid (m,n) \\ d > 1}} d^{k-3}e^{2\pi \frac{mn}{d^2}\sin \theta} \leq 4\sigma_0(m)m^{k-3}e^{\pi \frac{mn}{2}\sin \theta} \leq 4\sigma_0(m)m^{k-3} e^{\pi(2m  - 1)n\sin \theta}
	$$
	and
	$$
		1.008 \times 10^8 (mn)^{k-1} \sum_{\substack{d \mid (m,n) \\ d > 1}} d^{1-k} \leq 1.008 \times 10^8 \sigma_0(m)(mn)^{k-1}.
	$$
	Otherwise, $d \leq l \leq m  - \frac{1}{2}$ and we likewise obtain
	$$
		4\sum_{d \mid (l,n)} d^{k-3}e^{2\pi \frac{ln}{d^2}\sin \theta} \leq 4\sigma_0(m)m^{k-3}e^{\pi(2m  - 1)n\sin \theta}.
	$$
	and
	$$
		1.008 \times 10^8(ln)^{k-1}\sum_{d \mid (l,n)} d^{1-k} \leq 1.008 \times 10^8 \sigma_0(m)(mn)^{k-1}.
	$$
	Plugging these bounds into \eqref{TriR}, we have in both cases
	\begin{equation} \label{TriPlug}
		\sum_{\substack{d \mid (l,n) \\ (d,l) \neq (1,m )}} d^{k-1} \abs{\cF_{2-k,-\frac{ln}{d^2}}(e^{i\theta})} \leq 4\sigma_0(m)m^{k-3}e^{\pi(2m  - 1)n\sin \theta} + 1.008 \times 10^8 (mn)^{k-1}\sigma_0(m).
	\end{equation}
	
	Second, Lemma \ref{PBound} and the Deligne bound for coefficients of holomorphic cusp forms \cite{Deligne} yield
	$$
		\abs{a(n)\cF_{2-k,-l}(e^{i\theta})} \leq 4\sigma_0(n)n^{\frac{k-1}{2}}e^{2\pi l \sin \theta}+ 1.008 \times 10^8 \sigma_0(n)l^{k-1}n^{\frac{k-1}{2}}.
	$$
	Since $l \leq m$ and $\sigma_0(n) \leq 2\sqrt{n}$, we have
	$$
		\abs{a(n)\cF_{2-k,-l}(e^{i\theta})} \leq 4n^{\frac{k}{2}}e^{2\pi m \sin \theta}+ 1.008 \times 10^8m^{k-1}n^{\frac{k}{2}}.
	$$
	It is clear that $m^{k-1}n^{\frac{k}{2}} \leq (mn)^{k-1}$ for all $m,n \geq 2$. We then bound $\sigma_0(m)m^{k-3}$ and $n^{\frac{k}{2}}$ by the common bound of $m^{k-2}n^{\frac{k}{2}}$, which, together with \eqref{TriPlug} and the trivial bound $2m \leq (2m - 1)n$, grants us
	$$
		\abs{R_n(F;\theta)} \leq \left(  8m^{k-2}n^{\frac{k}{2}}e^{\pi(2m - 1)n\sin \theta} + 2.016 \times 10^8 (mn)^{k-1}\right) \sum_{l=1}^{m } \abs{c_F^+(-l)}.
	$$
\end{proof}

Altogether, we are now able to show \eqref{Goal}.
\begin{lemma} \label{QRBound}
	If $n \geq 7$ is such that
	$$
		C_F n^{k-1}e^{-\pi n \frac{\sqrt{3}}{2}} < 1,
	$$
	then we have
	$$
		\abs{e^{\frac{2-k}{2}i\theta}e^{-2\pi mn \sin \theta}H_n^\ast(F;e^{i\theta}) - f_{mn}(\theta)} < 2.
	$$
\end{lemma}

\begin{proof}
	By the triangle inequality applied to \eqref{QR}, we have
	$$
		\abs{e^{\frac{2-k}{2}i\theta}e^{-2\pi mn \sin \theta}H_n^\ast(F;e^{i\theta}) - f_{mn}(\theta)} \leq e^{-2\pi mn \sin \theta}\left(\abs{Q_{2-k,-mn}(\theta)} + \abs{R_n(F;\theta)}\right).
	$$
	Inserting the bounds from Lemma \ref{QBound}, we obtain
	\begin{equation} \label{Q}
		e^{-2\pi mn \sin \theta}\abs{Q_{2-k,-m  n}(\theta)} \leq 1 + 8mne^{-\pi m n \frac{\sqrt{3}}{2}} + 1.008 \times 10^{8}(mn)^{k-1}e^{-\pi m n \sqrt{3}}.
	\end{equation}
	Meanwhile, Lemma \ref{RBound} yields
	\begin{equation} \label{R}
		e^{-2\pi mn \sin \theta} \abs{R_n(F;\theta)} \leq \left( 8m^{k-2}n^{\frac{k}{2}} e^{-\pi n \frac{\sqrt{3}}{2}} + 2.016 \times 10^8 (mn)^{k-1} e^{-\pi mn \sqrt{3}} \right) \sum_{l=1}^{m } \abs{c_F^+(-l)}.
	\end{equation}

	Now let $y = mn$. Since $m^{k-1}n^{\frac{k}{2}} \leq (mn)^{k-1}$, we may reformulate the bound in \eqref{R} as
	$$
		8y^{k-1}e^{-\pi n \frac{\sqrt{3}}{2}} + 2.016 \times 10^8 y^{k-1} e^{-\pi y \sqrt{3}} \leq 8y^{k-1}e^{-\pi n \frac{\sqrt{3}}{2}} + 2.016 \times 10^8 y^{k-1} e^{-\pi y \frac{\sqrt{3}}{2}} e^{-\pi n \frac{\sqrt{3}}{2}}.
	$$
	For $y \geq 7$, we have
	\begin{equation} \label{y}
		2.016 \times 10^8 e^{-\pi y \frac{\sqrt{3}}{2}} \leq 2,
	\end{equation}
	whereby
	$$
		8y^{k-1}e^{-\pi n \frac{\sqrt{3}}{2}} + 2.016 \times 10^8 y^{k-1} e^{-\pi y \sqrt{3}} \leq 10y^{k-1}e^{-\pi n \frac{\sqrt{3}}{2}}.
	$$
	Similarly, we may rewrite \eqref{Q} using \eqref{y} as
	$$
		1 + 8ye^{-\pi y \frac{\sqrt{3}}{2}} + 1.008 \times 10^{8}y^{k-1}e^{-\pi y \sqrt{3}} \leq 1 + (y^{k-2} + 8)ye^{-\pi y \frac{\sqrt{3}}{2}} \leq 1 + 2y^{k-1}e^{-\pi y \frac{\sqrt{3}}{2}}.
	$$
	Recalling that
	$$
		C_F := \frac{1}{4}\max \left\{80m^{k-1}\sum_{l=1}^{m} \abs{c_F^+(-l)}, 1\right\},
	$$
	our assumption on $n$ ensures then that
	$$
		10y^{k-1}e^{-\pi n \frac{\sqrt{3}}{2}}\sum_{l=1}^{m } \abs{c_F^+(-l)} = \frac{1}{2}C_Fn^{k-1}e^{-\pi n \frac{\sqrt{3}}{2}} < \frac{1}{2}
	$$
	and
	$$
		2y^{k-1}e^{-\pi y \frac{\sqrt{3}}{2}} \leq \frac{1}{10C_F}n^{k-1}e^{-\pi n \frac{\sqrt{3}}{2}} < \frac{1}{10C_F} \leq \frac{2}{5}
	$$
	since we stipulate that $C_F \geq \frac{1}{4}$. In all, we obtain
	$$
		e^{-2\pi mn \sin \theta}\left(\abs{Q_{2-k,-m  n}(\theta)} + \abs{R_n(F;\theta)}\right) < 1 + \frac{2}{5} + \frac{1}{2} < 2
	$$
	as desired.
\end{proof}

\begin{proof}[Proof of Theorem~\ref{Theorem2}]
	Let $g_{mn}(\theta) := \frac{k-2}{2}\theta + 2\pi mn \cos \theta$, so that
	$$
		f_{mn}(\theta) = 2(1 - e^{-4\pi mn \sin \theta}e_{k-2}(4\pi mn \sin \theta)) \cos(g_{mn}(\theta)).
	$$
	The zeros and extrema of $f_{mn}(\theta)$ are controlled by $g_{mn}(\theta)$. One readily verifies by computation that
	$$
		1 - e^{-4\pi mn \sin \theta}e_{k-2}(4\pi mn \sin \theta) \geq 0.99,
	$$
	for all $k \geq 4$, as $e_{k-2}(x)$ is monotonic in $x$ and $k$. Thus, $\abs{f_{mn}(\theta)} \geq 1.98$ whenever $g_{mn}(\theta)$ is an integer multiple of $\pi$.
	
	We then observe that $g_{mn}(\theta)$ decreases from $\frac{(k-2)\pi}{6} + \pi mn$ to $\frac{(k-2)\pi}{4}$ on $[\frac{\pi}{3},\frac{\pi}{2}]$, hitting $mn + \lfloor (k-2)/6 \rfloor - \lceil (k-2)/4 \rceil$ consecutive integer multiples of $\pi$. Thus, Lemma~\ref{QRBound} implies that $H_n^\ast(F;e^{i\theta})$ changes sign at least this many times in $[\frac{\pi}{3},\frac{\pi}{2}]$, once in each subinterval of the form
	$$
		(g_{mn}^{-1}(\pi (i + 1)),g_{mn}^{-1}(\pi i))
	$$
	for $\lceil (k-2)/4 \rceil \leq i \leq mn + \lfloor (k-2)/6 \rfloor$. Therefore, by the Intermediate Value Theorem, $H_n^\ast(F;e^{i\theta})$ has a zero in each subinterval, and each of these zeros is a zero of $H_n(F;e^{i\theta})$.
	
	The remaining zeros of $H_n(F;e^{i\theta})$ arise from the factor of $\widetilde{E}_{k-2}(e^{i\theta})$. Indeed, we have
	$$
		\lfloor (k-2)/6 \rfloor - \lceil (k-2)/4 \rceil =
		\begin{cases}
			-2 - b(k - 2) & \text{if } k \equiv 0,4 \pmod{12}, \\
			-1 - b(k - 2) & \text{if } k \equiv 6,8,10 \pmod{12}, \\
			-b(k - 2) & \text{if } k \equiv 2 \pmod{12}.
		\end{cases}
	$$
	Likewise, letting $Z_{k-2}$ denote the zeros of $\tilde{E}_{k-2}(e^{i\theta})$ contained in $\cA$, we have
	$$
		Z_{k-2} = \begin{cases}
			\{\frac{\pi}{3},\frac{\pi}{2}\} & \text{if } k \equiv 0,4 \pmod{12}, \\
			\{\frac{\pi}{2}\} & \text{if } k \equiv 8 \pmod{12}, \\
			\{\frac{\pi}{3}\} & \text{if } k \equiv 6,10 \pmod{12}, \\
			\es & \text{if } k \equiv 2 \pmod{12}.
		\end{cases}
	$$
	In all cases, we see that $mn + \lfloor (k-2)/6 \rfloor - \lceil (k-2)/4 \rceil + \abs{Z_{k-2}} = mn - b(k - 2) = \deg P_n(F;x)$. Since $j \colon \cA \to [0, 1728]$ is a bijection, we have located all $mn - b(k - 2)$ zeros of $P_n(F;x)$, which are simple and in the interval $[0, 1728]$.
\end{proof}

\end{document}